\documentclass[12pt,fleqn]{article}

\usepackage{amscd,amsmath,amssymb,amsthm}
\usepackage{geometry} 
\usepackage[pdfpagemode=UseNone,pdfstartview=FitH,hypertex]{hyperref}
\usepackage[all]{xy} \CompileMatrices

\newcommand{\Ap}[1][]{A_p\, #1}
\newcommand{\abs}[1]{\left|#1\right|}
\newcommand{\Bp}[1][]{B_p\, #1}
\newcommand{\bdry}[1]{\partial #1}
\newcommand{\A}{{\cal A}}
\newcommand{\F}{{\cal F}}
\newcommand{\closure}[1]{\overline{#1}}
\newcommand{\comp}{\circ}
\newcommand{\dint}{\ds{\int}}
\newcommand{\dist}[2]{\text{dist}\, (#1,#2)}
\newcommand{\dnorm}[2][]{\left\|#2\right\|_{#1}^\ast}
\newcommand{\ds}[1]{\displaystyle #1}
\newcommand{\dualp}[3][]{\left(#2,#3\right)_{#1}}
\newcommand{\eps}{\varepsilon}
\newcommand{\half}{\frac{1}{2}}
\newcommand{\incl}{\hookrightarrow}
\newcommand{\isom}{\approx}
\newcommand{\loc}{\text{loc}}
\newcommand{\M}{{\cal M}}
\newcommand{\N}{\mathbb N}
\newcommand{\norm}[2][]{\left\|#2\right\|_{#1}}
\renewcommand{\o}{\text{o}}
\newcommand{\PS}[1]{$(\text{PS})_{#1}$}
\newcommand{\pnorm}[2][]{\if #1'' \left|#2\right|_p \else \left|#2\right|_{#1} \fi}
\newcommand{\QED}{\mbox{\qedhere}}
\newcommand{\R}{\mathbb R}
\newcommand{\RP}{\R \text{P}}
\newcommand{\restr}[2]{\left.#1\right|_{#2}}
\newcommand{\seq}[1]{\left(#1\right)}
\newcommand{\set}[1]{\left\{#1\right\}}
\newcommand{\wto}{\rightharpoonup}
\newcommand{\Z}{\mathbb Z}

\DeclareMathOperator{\divg}{div}
\DeclareMathOperator{\im}{im}

\newenvironment{enumroman}{\begin{enumerate}
\renewcommand{\theenumi}{$(\roman{enumi})$}
\renewcommand{\labelenumi}{$(\roman{enumi})$}}{\end{enumerate}}

\newenvironment{properties}[1]{\begin{enumerate}
\renewcommand{\theenumi}{$(#1_\arabic{enumi})$}
\renewcommand{\labelenumi}{$(#1_\arabic{enumi})$}}{\end{enumerate}}

\newtheorem{lemma}{Lemma}[section]
\newtheorem{proposition}[lemma]{Proposition}
\newtheorem{theorem}[lemma]{Theorem}

\theoremstyle{definition}
\newtheorem{definition}[lemma]{Definition}

\theoremstyle{remark}
\newtheorem{example}[lemma]{Example}
\newtheorem{remark}[lemma]{Remark}

\numberwithin{equation}{section}

\title{\bf An abstract critical point theorem with applications to elliptic problems with combined nonlinearities}
\author{\bf Kanishka Perera\\
Department of Mathematical Sciences\\
Florida Institute of Technology\\
Melbourne, FL 32901, USA\\
\em kperera@fit.edu}
\date{}

\begin{document}

\maketitle

\begin{abstract}
We prove an abstract critical point theorem based on a cohomological index theory that produces pairs of nontrivial critical points with nontrivial higher critical groups. This theorem yields pairs of nontrivial solutions that are neither local minimizers nor of mountain pass type for problems with combined nonlinearities. Applications are given to subcritical and critical $p$-Laplacian problems, Kirchhoff type nonlocal problems, and critical fractional $p$-Laplacian problems.
\end{abstract}

\section{Introduction and statement of results}

The purpose of this paper is to prove an abstract critical point theorem that can be used to obtain pairs of nontrivial solutions of problems of the type
\begin{equation} \label{3}
\left\{\begin{aligned}
- \Delta_p\, u & = \lambda\, |u|^{p-2}\, u + \mu f(x,u) + |u|^{q-2}\, u && \text{in } \Omega\\[10pt]
u & = 0 && \text{on } \bdry{\Omega},
\end{aligned}\right.
\end{equation}
where $\Omega$ is a bounded domain in $\R^N,\, N \ge 1$, $p > 1$, $\Delta_p\, u = \divg (|\nabla u|^{p-2}\, \nabla u)$ is the $p$-Laplacian of $u$, $p < q \le p^\ast = Np/(N - p)$ if $p < N$ and $p < q < \infty$ if $p \ge N$, $\lambda, \mu > 0$ are parameters, and $f$ is a Carath\'{e}odory function on $\Omega \times \R$ satisfying
\begin{equation} \label{18}
f(x,t) = |t|^{\sigma - 2}\, t + \o(|t|^{\sigma - 1}) \quad \text{as } t \to 0, \text{uniformly a.e.\! in } \Omega
\end{equation}
for some $1 < \sigma < p$, the sign condition
\begin{equation} \label{30}
F(x,t) = \int_0^t f(x,s)\, ds > 0 \quad \text{for a.a.\! } x \in \Omega \text{ and all } t \in \R \setminus \set{0},
\end{equation}
and the growth condition
\begin{equation} \label{27}
|f(x,t)| \le a\, (|t|^{r-1} + 1) \quad \text{for a.a.\! } x \in \Omega \text{ and all } t \in \R
\end{equation}
for some $a > 0$ and $p < r < q$. Denoting by $\lambda_1 > 0$ the first Dirichlet eigenvalue of $- \Delta_p$ on $\Omega$, the case where $0 < \lambda < \lambda_1$ and $\mu > 0$ is sufficiently small can be handled using a local minimization argument and the mountain pass theorem (see Ambrosetti et al.\! \cite{MR1276168}, Garc\'{\i}a Azorero et al.\! \cite{MR1776988}, Kyritsi and Papageorgiou \cite{MR2476670}, Papageorgiou and R\u{a}dulescu \cite{MR3367938}, and Furtado et al.\! \cite{MR3518343}). However, when $\lambda \ge \lambda_1$, the associated variational functional
\[
E(u) = \frac{1}{p} \int_\Omega |\nabla u|^p\, dx - \frac{\lambda}{p} \int_\Omega |u|^p\, dx - \mu \int_\Omega F(x,u)\, dx - \frac{1}{q} \int_\Omega |u|^q\, dx, \quad u \in W^{1,\,p}_0(\Omega)
\]
no longer has the mountain pass geometry and no multiplicity results are available in the literature. We will prove a linking theorem based on a cohomological index theory that can capture the geometry of this functional to produce a pair of nontrivial critical points for all $\lambda > 0$ and sufficiently small $\mu > 0$. These critical points are neither local minimizers nor of mountain pass type in general. They are higher critical points in the sense that they each have a nontrivial higher critical group (see Definition \ref{Definition 2}).

To state our main result, let $W$ be a Banach space and let $\M$ be a bounded symmetric subset of $W \setminus \set{0}$ radially homeomorphic to the unit sphere $S = \set{u \in W : \norm{u} = 1}$, i.e., the restriction to $\M$ of the radial projection $\pi : W \setminus \set{0} \to S,\, u \mapsto u/\norm{u}$ is a homeomorphism. Then the radial projection from $W \setminus \set{0}$ onto $\M$ is given by
\[
\pi_\M = (\restr{\pi}{\M})^{-1} \comp \pi.
\]
For a symmetric set $A \subset W \setminus \set{0}$, denote by $i(A)$ its $\Z_2$-cohomological index (see Definition \ref{Definition 1}). We have the following theorem.

\begin{theorem} \label{Theorem 1}
Let $E$ be a $C^1$-functional on $W$ and let $A_0$ and $B_0$ be disjoint closed symmetric subsets of $\M$ such that
\[
i(A_0) = i(\M \setminus B_0) = k < \infty.
\]
Assume that there exist $w_0 \in \M \setminus A_0$, $0 \le r < \rho < R$, and $a < b$ such that, setting
\begin{gather}
A_1 = \set{\pi_\M((1 - s)\, v + sw_0) : v \in A_0,\, 0 \le s \le 1}, \label{14}\\[10pt]
A^\ast = \set{tu : u \in A_1,\, r \le t \le R}, \notag\\[10pt]
B^\ast = \set{tw : w \in B_0,\, 0 \le t \le \rho}, \notag\\[10pt]
A = \set{ru : u \in A_1} \cup \set{tv : v \in A_0,\, r \le t \le R} \cup \set{Ru : u \in A_1}, \label{15}\\[10pt]
B = \set{\rho w : w \in B_0}, \label{16}
\end{gather}
we have
\begin{equation} \label{2}
a < \inf_{B^\ast}\, E, \qquad \sup_A\, E < \inf_B\, E, \qquad \sup_{A^\ast}\, E < b.
\end{equation}
If $E$ satisfies the {\em \PS{c}} condition for all $c \in (a,b)$, then $E$ has a pair of critical points $u_1, u_2$ with
\[
\inf_{B^\ast}\, E \le E(u_1) \le \sup_A\, E, \qquad \inf_B\, E \le E(u_2) \le \sup_{A^\ast}\, E.
\]
If, in addition, $E$ has only a finite number of critical points with the corresponding critical values in $(a,b)$, then $u_1$ and $u_2$ can be chosen to satisfy
\[
C^k(E,u_1) \ne 0, \qquad C^{k+1}(E,u_2) \ne 0.
\]
\end{theorem}

We will prove this theorem in Section \ref{Section 2}. The proof is based on the notion of a cohomological linking. In the course of the proof, we will also establish a new linking result of independent interest (see Theorem \ref{Theorem 3}).

Next we prove a multiplicity result for a class of abstract operator equations that includes problem \eqref{3} as a special case. Let $(W,\norm{\, \cdot\,})$ be a uniformly convex Banach space with dual $(W^\ast,\dnorm{\, \cdot\,})$ and duality pairing $\dualp{\cdot}{\cdot}$. Recall that $h \in C(W,W^\ast)$ is a potential operator if there is a functional $H \in C^1(W,\R)$, called a potential for $h$, such that $H' = h$. We consider the nonlinear operator equation
\begin{equation} \label{4}
\Ap[u] = \lambda \Bp[u] + \mu f(u) + g(u)
\end{equation}
in $W^\ast$, where $\Ap, \Bp, f, g \in C(W,W^\ast)$ are potential operators satisfying the following assumptions, and $\lambda, \mu > 0$ are parameters:
\begin{enumerate}
\item[$(A_1)$] $\Ap$ is $(p - 1)$-homogeneous and odd for some $p \in (1,\infty)$: $\Ap[(tu)] = |t|^{p-2}\, t\, \Ap[u]$ for all $u \in W$ and $t \in \R$,
\item[$(A_2)$] $\dualp{\Ap[u]}{v} \le \norm{u}^{p-1} \norm{v}$ for all $u, v \in W$, and equality holds if and only if $\alpha u = \beta v$ for some $\alpha, \beta \ge 0$, not both zero (in particular, $\dualp{\Ap[u]}{u} = \norm{u}^p$ for all $u \in W$),
\item[$(B_1)$] $\Bp$ is $(p - 1)$-homogeneous and odd: $\Bp[(tu)] = |t|^{p-2}\, t\, \Bp[u]$ for all $u \in W$ and $t \in \R$,
\item[$(B_2)$] $\dualp{\Bp[u]}{u} > 0$ for all $u \in W \setminus \set{0}$, and $\dualp{\Bp[u]}{v} \le \dualp{\Bp[u]}{u}^{(p-1)/p} \dualp{\Bp[v]}{v}^{1/p}$ for all $u, v \in W$,
\item[$(B_3)$] $\Bp$ is a compact operator,
\item[$(F_1)$] the potential $F$ of $f$ with $F(0) = 0$ satisfies $\ds{\lim_{t \to 0}}\, \dfrac{F(tu)}{|t|^p} = + \infty$ uniformly on compact subsets of $W \setminus \set{0}$,
\item[$(F_2)$] $F(u) > 0$ for all $u \in W \setminus \set{0}$,
\item[$(F_3)$] $F$ is bounded on bounded subsets of $W$,
\item[$(G_1)$] the potential $G$ of $g$ with $G(0) = 0$ satisfies $G(u) = \o(\norm{u}^p)$ as $u \to 0$,
\item[$(G_2)$] $G(u) > 0$ for all $u \in W \setminus \set{0}$,
\item[$(G_3)$] $G$ is bounded on bounded subsets of $W$,
\item[$(G_4)$] $\ds{\lim_{t \to + \infty}} \dfrac{G(tu)}{t^p} = + \infty$ uniformly on compact subsets of $W \setminus \set{0}$.
\end{enumerate}

Solutions of equation \eqref{4} coincide with critical points of the $C^1$-functional
\begin{equation} \label{22}
E(u) = I_p(u) - \lambda J_p(u) - \mu F(u) - G(u), \quad u \in W,
\end{equation}
where
\[
I_p(u) = \frac{1}{p} \dualp{\Ap[u]}{u} = \frac{1}{p} \norm{u}^p, \qquad J_p(u) = \frac{1}{p} \dualp{\Bp[u]}{u}
\]
are the potentials of $\Ap$ and $\Bp$ satisfying $I_p(0) = 0 = J_p(0)$, respectively (see Proposition \ref{Proposition 1}). First we prove the following theorem, which assumes that $E$ satisfies the \PS{} condition.

\begin{theorem} \label{Theorem 2}
Assume that $(A_1)$--$(G_4)$ hold and $E$ satisfies the {\em \PS{c}} condition for all $c \in \R$. If $\lambda > 0$, then $\exists\, \mu_0 > 0$ such that equation \eqref{4} has two nontrivial solutions $u_1, u_2$ with
\[
E(u_1) < 0 < E(u_2)
\]
for $0 < \mu < \mu_0$.
\end{theorem}

We will prove Theorem \ref{Theorem 2} in Section \ref{Section 3}. The proof will involve the nonlinear eigenvalue problem
\begin{equation} \label{5}
\Ap[u] = \lambda \Bp[u].
\end{equation}
Let $\M = \set{u \in W : I_p(u) = 1}$. Then $\M \subset W \setminus \set{0}$ is a bounded complete symmetric $C^1$-Finsler manifold radially homeomorphic to the unit sphere in $W$, and eigenvalues of problem \eqref{5} coincide with critical values of the $C^1$-functional
\[
\Psi(u) = \frac{1}{J_p(u)}, \quad u \in \M.
\]
Denote by $\F$ the class of symmetric subsets of $\M$ and by $i(M)$ the cohomological index of $M \in \F$, let $\F_k = \set{M \in \F : i(M) \ge k}$, and set
\[
\lambda_k := \inf_{M \in \F_k}\, \sup_{u \in M}\, \Psi(u), \quad k \in \N.
\]
Then
\[
\lambda_1 = \inf_{u \in \M}\, \Psi(u) > 0
\]
is the first eigenvalue of problem \eqref{5} and $\lambda_1 \le \lambda_2 \le \cdots$ is an unbounded sequence of eigenvalues. Moreover, denoting by $\Psi^a = \set{u \in \M : \Psi(u) \le a}$ (resp. $\Psi_a = \set{u \in \M : \Psi(u) \ge a}$) the sublevel (resp. superlevel) sets of $\Psi$, we have
\begin{equation} \label{6}
\lambda_k < \lambda_{k+1} \implies i(\Psi^{\lambda_k}) = i(\M \setminus \Psi_{\lambda_{k+1}}) = k
\end{equation}
(see Perera et al.\! \cite[Theorem 4.6]{MR2640827}). An intermediate step in the proof of Theorem \ref{Theorem 2} is the following result of independent interest.

\begin{theorem} \label{Theorem 4}
Assume that $(A_1)$--$(B_3)$ hold. If $\lambda_k < \lambda_{k+1}$, then the sublevel set $\Psi^{\lambda_k}$ has a compact symmetric subset of index $k$.
\end{theorem}

\begin{remark}
Eigenvalues based on the cohomological index were first introduced in Perera \cite{MR1998432}.
\end{remark}

Theorem \ref{Theorem 2} can be used to obtain a pair of nontrivial solutions of the $p$-Laplacian problem \eqref{3} for all $\lambda > 0$ and sufficiently small $\mu > 0$ in the subcritical case. More precisely, we have the following multiplicity result.

\begin{theorem}
Let $1 < p < q$, with $q < p^\ast$ if $p < N$ and $q < \infty$ if $p \ge N$, let $\lambda > 0$, and let $f$ be a Carath\'{e}odory function on $\Omega \times \R$ satisfying \eqref{18}--\eqref{27} for some $1 < \sigma < p < r < q$. Then $\exists\, \mu_0 > 0$ such that problem \eqref{3} has two nontrivial solutions $u_1, u_2$ with
\[
E(u_1) < 0 < E(u_2)
\]
for $0 < \mu < \mu_0$.
\end{theorem}

This theorem follows from Theorem \ref{Theorem 2} with $W = W^{1,\,p}_0(\Omega)$ and the operators $\Ap, \Bp, f, g \linebreak \in C(W^{1,\,p}_0(\Omega),W^{-1,\,p'}(\Omega))$ given by
\begin{multline*}
\dualp{\Ap[u]}{v} = \int_\Omega |\nabla u|^{p-2}\, \nabla u \cdot \nabla v\, dx, \quad \dualp{\Bp[u]}{v} = \int_\Omega |u|^{p-2}\, uv\, dx,\\[5pt]
\dualp{f(u)}{v} = \int_\Omega f(x,u)\, v\, dx, \quad \dualp{g(u)}{v} = \int_\Omega |u|^{q-2}\, uv\, dx, \quad u, v \in W^{1,\,p}_0(\Omega).
\end{multline*}
It is easily seen that $(A_1)$--$(G_4)$ hold.

As another application of Theorem \ref{Theorem 2}, consider the Kirchhoff type nonlocal problem
\begin{equation} \label{19}
\left\{\begin{aligned}
- \left(\int_\Omega |\nabla u|^2\, dx\right) \Delta u & = \lambda u^3 + \mu f(x,u) + |u|^{q-2}\, u && \text{in } \Omega\\[10pt]
u & = 0 && \text{on } \bdry{\Omega},
\end{aligned}\right.
\end{equation}
where $\Omega$ is a bounded domain in $\R^N,\, N = 1, 2, \text{or } 3$, $4 < q < 6$ if $N = 3$ and $4 < q < \infty$ if $N = 1 \text{ or } 2$, $\lambda, \mu > 0$ are parameters, and $f$ is a Carath\'{e}odory function on $\Omega \times \R$ satisfying \eqref{18}--\eqref{27} for some $1 < \sigma < 4 < r < q$. Solutions of this problem coincide with critical points of the functional
\[
E(u) = \frac{1}{4} \left(\int_\Omega |\nabla u|^2\, dx\right)^2 - \frac{\lambda}{4} \int_\Omega u^4\, dx - \mu \int_\Omega F(x,u)\, dx - \frac{1}{q} \int_\Omega |u|^q\, dx, \quad u \in H^1_0(\Omega),
\]
where $F(x,t) = \int_0^t f(x,s)\, ds$. Applying Theorem \ref{Theorem 2} with $W = H^1_0(\Omega)$ and the operators $\Ap, \Bp, f, g \in C(H^1_0(\Omega),H^{-1}(\Omega))$ given by
\begin{multline*}
\dualp{\Ap[u]}{v} = \left(\int_\Omega |\nabla u|^2\, dx\right) \int_\Omega \nabla u \cdot \nabla v\, dx, \quad \dualp{\Bp[u]}{v} = \int_\Omega u^3 v\, dx,\\[5pt]
\dualp{f(u)}{v} = \int_\Omega f(x,u)\, v\, dx, \quad \dualp{g(u)}{v} = \int_\Omega |u|^{q-2}\, uv\, dx, \quad u, v \in H^1_0(\Omega),
\end{multline*}
we have the following multiplicity result.

\begin{theorem}
Let $q > 4$, with $q < 6$ if $N = 3$ and $q < \infty$ if $N = 1 \text{ or } 2$, let $\lambda > 0$, and let $f$ be a Carath\'{e}odory function on $\Omega \times \R$ satisfying \eqref{18}--\eqref{27} for some $1 < \sigma < 4 < r < q$. Then $\exists\, \mu_0 > 0$ such that problem \eqref{19} has two nontrivial solutions $u_1, u_2$ with
\[
E(u_1) < 0 < E(u_2)
\]
for $0 < \mu < \mu_0$.
\end{theorem}

Finally we prove a variant of Theorem \ref{Theorem 2} that only assumes a local \PS{} condition and hence applicable to critical growth problems.

\begin{theorem} \label{Theorem 5}
Assume that $(A_1)$--$(G_4)$ hold and $\exists\, c_\mu > 0$ such that $E$ satisfies the {\em \PS{c}} condition for all $c < c_\mu$. If $0 < \lambda < \lambda_1$, assume that $\exists\, w_0 \in \M$ such that
\begin{equation} \label{23}
\sup_{t \ge 0}\, E(tw_0) < c_\mu
\end{equation}
for all sufficiently small $\mu > 0$. If $\lambda_k \le \lambda < \lambda_{k+1}$, assume that there exist a compact symmetric subset $C$ of $\Psi^\lambda$ with $i(C) = k$ and $w_0 \in \M \setminus C$ such that
\begin{equation} \label{24}
\sup_{v \in C,\, s, t \ge 0}\, E(sv + tw_0) < c_\mu
\end{equation}
for all sufficiently small $\mu > 0$. Then $\exists\, \mu_0 > 0$ such that equation \eqref{4} has two nontrivial solutions $u_1, u_2$ with
\[
E(u_1) < 0 < E(u_2) < c_\mu
\]
for $0 < \mu < \mu_0$.
\end{theorem}

We will prove Theorem \ref{Theorem 5} in Section \ref{Section 3}. As an application, consider the critical $p$-Laplacian problem
\begin{equation} \label{25}
\left\{\begin{aligned}
- \Delta_p\, u & = \lambda\, |u|^{p-2}\, u + \mu f(x,u) + |u|^{p^\ast - 2}\, u && \text{in } \Omega\\[10pt]
u & = 0 && \text{on } \bdry{\Omega},
\end{aligned}\right.
\end{equation}
where $\Omega$ is a bounded domain in $\R^N$, $1 < p < N$, $\lambda, \mu > 0$ are parameters, and $f$ is a Carath\'{e}odory function on $\Omega \times \R$ satisfying \eqref{18}--\eqref{27} for some $1 < \sigma < p < r < p^\ast$. Solutions of this problem coincide with critical points of the functional
\[
E(u) = \frac{1}{p} \int_\Omega |\nabla u|^p\, dx - \frac{\lambda}{p} \int_\Omega |u|^p\, dx - \mu \int_\Omega F(x,u)\, dx - \frac{1}{p^\ast} \int_\Omega |u|^{p^\ast} dx, \quad u \in W^{1,\,p}_0(\Omega).
\]
Denote by $\sigma(- \Delta_p)$ the Dirichlet spectrum of $- \Delta_p$ in $\Omega$, consisting of those $\lambda \in \R$ for which the problem
\[
\left\{\begin{aligned}
- \Delta_p\, u & = \lambda\, |u|^{p-2}\, u && \text{in } \Omega\\[10pt]
u & = 0 && \text{on } \bdry{\Omega}
\end{aligned}\right.
\]
has a nontrivial solution. Let
\begin{equation} \label{28}
S = \inf_{u \in W^{1,\,p}_0(\Omega) \setminus \set{0}}\, \frac{\dint_\Omega |\nabla u|^p\, dx}{\left(\dint_\Omega |u|^{p^\ast} dx\right)^{p/p^\ast}}
\end{equation}
be the best Sobolev constant. We will prove the following multiplicity result in Section \ref{Section 4}.

\begin{theorem} \label{Theorem 6}
Let $p > 1$, let $N \ge p^2$, let $\lambda > 0$ with $\lambda \notin \sigma(- \Delta_p)$, and let $f$ be a Carath\'{e}odory function on $\Omega \times \R$ satisfying \eqref{18}--\eqref{27} for some $1 < \sigma < p < r < p^\ast$. Then $\exists\, \mu_0, \kappa > 0$ such that problem \eqref{25} has two nontrivial solutions $u_1, u_2$ with
\[
E(u_1) < 0 < E(u_2) < \frac{1}{N}\, S^{N/p} - \kappa \mu
\]
for $0 < \mu < \mu_0$.
\end{theorem}

\begin{remark}
When $\mu = 0$, one nontrivial solution was obtained in Garc{\'{\i}}a Azorero and Peral Alonso \cite{MR912211}, Egnell \cite{MR956567}, Guedda and V{\'e}ron \cite{MR1009077}, Arioli and Gazzola \cite{MR1741848}, and Degiovanni and Lancelotti \cite{MR2514055}.
\end{remark}

\hspace{-1pt} As another application of Theorem \ref{Theorem 5}, consider the critical fractional $p$-Laplacian problem
\begin{equation} \label{35}
\left\{\begin{aligned}
(- \Delta)_p^s\, u & = \lambda\, |u|^{p-2}\, u + \mu f(x,u) + |u|^{p_s^\ast - 2}\, u && \text{in } \Omega\\[10pt]
u & = 0 && \text{in } \R^N \setminus \Omega,
\end{aligned}\right.
\end{equation}
where $\Omega$ is a bounded domain in $\R^N$ with Lipschitz boundary, $s \in (0,1)$, $1 < p < N/s$, $(- \Delta)_p^s$ is the fractional $p$-Laplacian operator defined on smooth functions by
\[
(- \Delta)_p^s\, u(x) = 2 \lim_{\eps \searrow 0} \int_{\R^N \setminus B_\eps(x)} \frac{|u(x) - u(y)|^{p-2}\, (u(x) - u(y))}{|x - y|^{N+sp}}\, dy, \quad x \in \R^N,
\]
$p_s^\ast = Np/(N - sp)$ is the fractional critical Sobolev exponent, $\lambda, \mu > 0$ are parameters, and $f$ is a Carath\'{e}odory function on $\Omega \times \R$ satisfying \eqref{18}--\eqref{27} for some $1 < \sigma < p < r < p_s^\ast$. Let $\pnorm{\cdot}$ denote the norm in $L^p(\R^N)$, let
\[
[u]_{s,\,p} = \left(\int_{\R^{2N}} \frac{|u(x) - u(y)|^p}{|x - y|^{N+sp}}\, dx dy\right)^{1/p}
\]
be the Gagliardo seminorm of a measurable function $u : \R^N \to \R$, and let
\[
W^{s,\,p}(\R^N) = \set{u \in L^p(\R^N) : [u]_{s,\,p} < \infty}
\]
be the fractional Sobolev space endowed with the norm
\[
\norm[s,\,p]{u} = \left(\pnorm{u}^p + [u]_{s,\,p}^p\right)^{1/p}.
\]
We work in the closed linear subspace
\[
W^{s,\,p}_0(\Omega) = \set{u \in W^{s,\,p}(\R^N) : u = 0 \text{ a.e.\! in } \R^N \setminus \Omega},
\]
equivalently renormed by setting $\norm{\cdot} = [\cdot]_{s,\,p}$. Solutions of problem \eqref{35} coincide with critical points of the functional
\begin{multline*}
E(u) = \frac{1}{p} \int_{\R^{2N}} \frac{|u(x) - u(y)|^p}{|x - y|^{N+sp}}\, dx dy - \frac{\lambda}{p} \int_\Omega |u|^p\, dx - \mu \int_\Omega F(x,u)\, dx - \frac{1}{p_s^\ast} \int_\Omega |u|^{p_s^\ast}\, dx,\\[7.5pt]
u \in W^{s,\,p}_0(\Omega).
\end{multline*}
Denote by $\sigma((- \Delta)_p^s)$ the Dirichlet spectrum of $(- \Delta)_p^s$ in $\Omega$, consisting of those $\lambda \in \R$ for which the problem
\[
\left\{\begin{aligned}
(- \Delta)_p^s\, u & = \lambda\, |u|^{p-2}\, u && \text{in } \Omega\\[10pt]
u & = 0 && \text{in } \R^N \setminus \Omega
\end{aligned}\right.
\]
has a nontrivial solution. Let
\[
\dot{W}^{s,\,p}(\R^N) = \set{u \in L^{p_s^\ast}(\R^N) : [u]_{s,\,p} < \infty}
\]
endowed with the norm $\norm{\cdot}$, and let
\begin{equation} \label{36}
S = \inf_{u \in \dot{W}^{s,\,p}(\R^N) \setminus \set{0}}\, \frac{\dint_{\R^{2N}} \frac{|u(x) - u(y)|^p}{|x - y|^{N+sp}}\, dx dy}{\left(\dint_{\R^N} |u|^{p_s^\ast}\, dx\right)^{p/p_s^\ast}}
\end{equation}
be the best fractional Sobolev constant. We will prove the following multiplicity result in Section \ref{Section 5}.

\begin{theorem} \label{Theorem 7}
Let $p > 1$, let $N > sp^2$, let $\lambda > 0$ with $\lambda \notin \sigma((- \Delta)_p^s)$, and let $f$ be a Carath\'{e}odory function on $\Omega \times \R$ satisfying \eqref{18}--\eqref{27} for some $1 < \sigma < p < r < p_s^\ast$. Then $\exists\, \mu_0, \kappa > 0$ such that problem \eqref{35} has two nontrivial solutions $u_1, u_2$ with
\[
E(u_1) < 0 < E(u_2) < \frac{s}{N}\, S^{N/sp} - \kappa \mu
\]
for $0 < \mu < \mu_0$.
\end{theorem}

\begin{remark}
When $\mu = 0$, one nontrivial solution was obtained in Mosconi et al.\! \cite{MR3530213}.
\end{remark}

\section{Proof of Theorem \ref{Theorem 1}} \label{Section 2}

In this section we prove Theorem \ref{Theorem 1}. We begin by recalling the definition and some properties of the $\Z_2$-cohomological index of Fadell and Rabinowitz \cite{MR0478189}.

\begin{definition} \label{Definition 1}
Let $W$ be a Banach space and let $\A$ denote the class of symmetric subsets of $W \setminus \set{0}$. For $A \in \A$, let $\overline{A} = A/\Z_2$ be the quotient space of $A$ with each $u$ and $-u$ identified, let $f : \overline{A} \to \RP^\infty$ be the classifying map of $\overline{A}$, and let $f^\ast : H^\ast(\RP^\infty) \to H^\ast(\overline{A})$ be the induced homomorphism of the Alexander-Spanier cohomology rings. The cohomological index of $A$ is defined by
\[
i(A) = \begin{cases}
0 & \text{if } A = \emptyset\\[7.5pt]
\sup \set{m \ge 1 : f^\ast(\omega^{m-1}) \ne 0} & \text{if } A \ne \emptyset,
\end{cases}
\]
where $\omega \in H^1(\RP^\infty)$ is the generator of the polynomial ring $H^\ast(\RP^\infty) = \Z_2[\omega]$.
\end{definition}

\begin{example}
The classifying map of the unit sphere $S^{m-1}$ in $\R^m,\, m \ge 1$ is the inclusion $\RP^{m-1} \incl \RP^\infty$, which induces isomorphisms on the cohomology groups $H^q$ for $q \le m - 1$, so $i(S^{m-1}) = m$.
\end{example}

The following proposition summarizes the basic properties of this index.

\begin{proposition}[\cite{MR0478189}] \label{Proposition 3}
The index $i : \A \to \N \cup \set{0,\infty}$ has the following properties:
\begin{properties}{i}
\item Definiteness: $i(A) = 0$ if and only if $A = \emptyset$.
\item Monotonicity: If there is an odd continuous map from $A$ to $B$ (in particular, if $A \subset B$), then $i(A) \le i(B)$. Thus, equality holds when the map is an odd homeomorphism.
\item Dimension: $i(A) \le \dim W$.
\item Continuity: If $A$ is closed, then there is a closed neighborhood $N \in \A$ of $A$ such that $i(N) = i(A)$. When $A$ is compact, $N$ may be chosen to be a $\delta$-neighborhood $N_\delta(A) = \set{u \in W : \dist{u}{A} \le \delta}$.
\item Subadditivity: If $A$ and $B$ are closed, then $i(A \cup B) \le i(A) + i(B)$.
\item Stability: If $SA$ is the suspension of $A \ne \emptyset$, obtained as the quotient space of $A \times [-1,1]$ with $A \times \set{1}$ and $A \times \set{-1}$ collapsed to different points, then $i(SA) = i(A) + 1$.
\item Piercing property: If $A$, $A_0$ and $A_1$ are closed, and $\varphi : A \times [0,1] \to A_0 \cup A_1$ is a continuous map such that $\varphi(-u,t) = - \varphi(u,t)$ for all $(u,t) \in A \times [0,1]$, $\varphi(A \times [0,1])$ is closed, $\varphi(A \times \set{0}) \subset A_0$ and $\varphi(A \times \set{1}) \subset A_1$, then $i(\varphi(A \times [0,1]) \cap A_0 \cap A_1) \ge i(A)$.
\item Neighborhood of zero: If $U$ is a bounded closed symmetric neighborhood of $0$, then $i(\bdry{U}) = \dim W$.
\end{properties}
\end{proposition}

Next we recall the definition of a cohomological linking.

\begin{definition}
Let $A$ and $B$ be disjoint nonempty subsets of a Banach space $W$. We say that $A$ links $B$ cohomologically in dimension $k \ge 0$ if the inclusion $\iota : A \incl W \setminus B$ induces a nontrivial homomorphism $\iota^\ast : \widetilde{H}^k(W \setminus B) \to \widetilde{H}^k(A)$ between reduced cohomology groups in dimension $k$.
\end{definition}

\begin{example} \label{Example 1}
Let $\M$ be as in the introduction and let $A_0$ and $B_0$ be disjoint nonempty closed symmetric subsets of $\M$ such that
\[
i(A_0) = i(\M \setminus B_0) = k < \infty.
\]
Then for any $R > 0$, the set $\set{Rv : v \in A_0}$ links the set $\set{tw : w \in B_0,\, t \ge 0}$ cohomologically in dimension $k - 1$ (see Perera et al.\! \cite[Proposition 3.26]{MR2640827}).
\end{example}

In preparation for the proof of Theorem \ref{Theorem 1}, we prove the following linking result, which is of independent interest.

\begin{theorem} \label{Theorem 3}
Let $A_0$ and $B_0$ be disjoint closed symmetric subsets of $\M$ such that
\begin{equation} \label{1}
i(A_0) = i(\M \setminus B_0) = k < \infty.
\end{equation}
Let $w_0 \in \M \setminus A_0$ and $0 \le r < \rho < R$. Set $A_1 = \set{\pi_\M((1 - s)\, v + sw_0) : v \in A_0,\, 0 \le s \le 1}$ and
\[
A = \set{ru : u \in A_1} \cup \set{tv : v \in A_0,\, r \le t \le R} \cup \set{Ru : u \in A_1}, \quad B = \set{\rho w : w \in B_0}.
\]
Then $A$ links $B$ cohomologically in dimension $k$.
\end{theorem}

\begin{proof}
First we note that $A_1$ is contractible. Indeed, the mapping
\[
A_1 \times [0,1] \to A_1, \quad (u,t) \mapsto \pi_\M((1 - t)\, u + tw_0)
\]
is a contraction of $A_1$ to $w_0$.

Let
\begin{gather*}
A_2 = \set{ru : u \in A_1} \cup \set{tv : v \in A_0,\, r \le t \le R}, \quad A_3 = \set{Rv : v \in A_0},\\[10pt]
B_1 = \set{tw : w \in B_0,\, t \ge 0}, \quad B_2 = \set{tw : w \in B_0,\, t \ge \rho}.
\end{gather*}
The set $A_2$ is contractible. Indeed, the mapping
\[
A_2 \times [0,1] \to A_2, \quad (u,t) \mapsto (1 - t)\, u + tr\, \pi_\M(u)
\]
is a strong deformation retraction of $A_2$ onto $\set{ru : u \in A_1}$, which is homeomorphic to $A_1$ and hence contractible. Consider the commutative diagram
\[
\begin{CD}
\widetilde{H}^{k-1}(E \setminus B_2) @>>> \widetilde{H}^{k-1}(E \setminus B_1) @>>> H^k(E \setminus B_2,E \setminus B_1) @>>> \widetilde{H}^k(E \setminus B_2)\\
@VVV @VV{\iota_1^\ast}V @VV{\iota_2^\ast}V @VVV\\
\widetilde{H}^{k-1}(A_2) @>>> \widetilde{H}^{k-1}(A_3) @>\delta>> H^k(A_2,A_3) @>>> \widetilde{H}^k(A_2)
\end{CD}
\]
where the rows come from the exact sequences of the pairs $(A_2,A_3) \incl (E \setminus B_2,E \setminus B_1)$. By Example \ref{Example 1}, $A_3$ links $B_1$ cohomologically in dimension $k - 1$ and hence $\iota_1^\ast \ne 0$. Since $A_2$ is contractible, $\widetilde{H}^\ast(A_2) = 0$ and hence $\delta$ is an isomorphism by the exactness of the bottom row. So it follows from the commutativity of the middle square that $\iota_2^\ast \ne 0$.

Next let
\[
A_4 = \set{Ru : u \in A_1},
\]
and consider the commutative diagram
\[
\begin{CD}
H^k(E \setminus B,E \setminus B^\ast) @>\isom>> H^k(E \setminus B_2,E \setminus B_1)\\
@VV{\iota_3^\ast}V @VV{\iota_2^\ast}V\\
H^k(A,A_4) @>>> H^k(A_2,A_3)
\end{CD}
\]
induced by inclusions. The top arrow is an isomorphism by the excision property since $\set{tw : w \in B_0,\, t > \rho}$ is a closed subset of $E \setminus B$ contained in the open subset $E \setminus B^\ast$. Since $\iota_2^\ast \ne 0$, it follows that $\iota_3^\ast \ne 0$.

Finally consider the commutative diagram
\[
\begin{CD}
\widetilde{H}^{k-1}(E \setminus B^\ast) @>>> H^k(E \setminus B,E \setminus B^\ast) @>>> \widetilde{H}^k(E \setminus B) @>>> \widetilde{H}^k(E \setminus B^\ast)\\
@VVV @VV{\iota_3^\ast}V @VV{\iota^\ast}V @VVV\\
\widetilde{H}^{k-1}(A_4) @>>> H^k(A,A_4) @>j^\ast>> \widetilde{H}^k(A) @>>>\widetilde{H}^k(A_4)
\end{CD}
\]
where the rows come from the exact sequences of the pairs $(A,A_4) \incl (E \setminus B,E \setminus B^\ast)$. Since $A_4$ is homeomorphic to $A_1$ and hence contractible, $\widetilde{H}^\ast(A_4) = 0$. So $j^\ast$ is an isomorphism by the exactness of the bottom row. Since $\iota_3^\ast \ne 0$, it follows from the commutativity of the middle square that $\iota^\ast \ne 0$.
\end{proof}

Before proceeding to the proof of Theorem \ref{Theorem 1}, we recall the definition of critical groups.

\begin{definition} \label{Definition 2}
Let $E$ be a $C^1$-functional on a Banach space $W$ and let $u_0$ be an isolated critical point of $E$. The critical groups of $E$ at $u_0$ are defined by
\[
C^q(E,u_0) = H^q(E^c,E^c \setminus \set{u_0}), \quad q \ge 0,
\]
where $c = E(u_0)$ is the corresponding critical value and $E^c = \set{u \in W : E(u) \le c}$ is the sublevel set.
\end{definition}

If $E$ satisfies the \PS{c} condition for all $c \in [a,b]$ and $H^q(E^b,E^a) \ne 0$ for some $q \ge 0$, then $E$ has a critical point $u_0$ with $a \le E(u_0) \le b$. If, in addition, $a$ and $b$ are regular values and $E$ has only a finite number of critical points with the corresponding critical values in $(a,b)$, then $u_0$ can be chosen so that $C^q(E,u_0) \ne 0$ (see, e.g., Perera et al.\! \cite[Proposition 3.13]{MR2640827}). We will make use of this fact in the proof of Theorem \ref{Theorem 1}.

\begin{proof}[Proof of Theorem \ref{Theorem 1}]
For the sake of simplicity we only consider the case where $W$ is infinite dimensional. Since $B^\ast \cap A$ and $B \cap A^\ast$ are nonempty, $\inf E(B^\ast) \le \sup E(A)$ and $\inf E(B) \le \sup E(A^\ast)$. We will show that if $\alpha < \beta < \gamma$ satisfy
\[
a < \alpha < \inf_{B^\ast}\, E, \qquad \sup_A\, E < \beta < \inf_B\, E, \qquad \sup_{A^\ast}\, E < \gamma < b,
\]
then
\[
H^k(E^\beta,E^\alpha) \ne 0, \qquad H^{k+1}(E^\gamma,E^\beta) \ne 0.
\]
Then $E$ has a pair of critical points $u_1, u_2$ with
\[
\alpha \le E(u_1) \le \beta, \qquad \beta \le E(u_2) \le \gamma.
\]
If $\inf E(B^\ast)$ or $\sup E(A)$ is a critical value of $E$, we can take $u_1$ to be at one of those levels. Otherwise, $\alpha$ and $\beta$ can be chosen so that $E$ has no critical values in $[\alpha,\inf E(B^\ast)] \cup [\sup E(A),\beta]$ since $E$ satisfies the \PS{} condition at those levels, and then
\[
\inf_{B^\ast}\, E < E(u_1) < \sup_A\, E.
\]
Similarly, $u_2$ can be chosen to satisfy
\[
\inf_B\, E \le E(u_2) \le \sup_{A^\ast}\, E.
\]
If $E$ has only a finite number of critical points with the corresponding critical values in $(a,b)$, then $\alpha$ and $\beta$ can be chosen so that $E$ has no critical values in $[\alpha,\inf E(B^\ast)) \cup (\sup E(A),\beta]$ and
\[
C^k(E,u_1) \ne 0, \qquad C^{k+1}(E,u_2) \ne 0.
\]

By Theorem \ref{Theorem 3}, $A$ links $B$ cohomologically in dimension $k$ and hence the inclusion $\iota : A \incl W \setminus B$ induces a nontrivial homomorphism $\iota^\ast : \widetilde{H}^k(W \setminus B) \to \widetilde{H}^k(A)$. Since $E < \beta$ on $A$ and $E > \beta$ on $B$, we also have the inclusions $\iota_1 : A \incl E^\beta$ and $\iota_2 : E^\beta \incl W \setminus B$, which induce the commutative diagram
\[
\xymatrix{{\widetilde{H}}^k(W \setminus B) \ar[d]^{\iota_2^\ast} \ar[dr]^{\iota^\ast}\\
{\widetilde{H}}^k(E^\beta) \ar[r]^{\iota_1^\ast} & {\widetilde{H}}^k(A).}
\]
This gives
\[
\iota_1^\ast \iota_2^\ast = \iota^\ast \ne 0,
\]
so both $\iota_1^\ast$ and $\iota_2^\ast$ are nontrivial homomorphisms.

First we show that $H^k(E^\beta,E^\alpha) \ne 0$. Since $E > \alpha$ on $B^\ast$ and $\alpha < \beta$, we have the inclusions $E^\alpha \incl W \setminus B^\ast \incl W \setminus B$ and $E^\alpha \incl E^\beta \incl W \setminus B$, yielding the commutative diagram
\[
\begin{CD}
\widetilde{H}^k(W \setminus B) @>>> \widetilde{H}^k(W \setminus B^\ast)\\
@VV{\iota_2^\ast}V @VVV\\
\widetilde{H}^k(E^\beta) @>i^\ast>> \widetilde{H}^k(E^\alpha).
\end{CD}
\]
The set $W \setminus B^\ast$ is contractible. Indeed, for any $\rho' > \rho$, the mapping
\[
(W \setminus B^\ast) \times [0,1] \to W \setminus B^\ast, \quad (u,t) \mapsto (1 - t)\, u + t \rho'\, \pi_\M(u)
\]
is a strong deformation retraction of $W \setminus B^\ast$ onto $\set{\rho' u : u \in \M}$, which is homeomorphic to $S$ and hence contractible since $W$ is infinite dimensional. So $\widetilde{H}^\ast(W \setminus B^\ast) = 0$ and hence
\[
i^\ast \iota_2^\ast = 0.
\]
Since $\iota_2^\ast$ is nontrivial, this implies that $i^\ast$ is not injective. Now consider the exact sequence of the pair $(E^\beta,E^\alpha)$:
\[
\begin{CD}
\cdots @>\delta>> H^k(E^\beta,E^\alpha) @>j^\ast>> \widetilde{H}^k(E^\beta) @>i^\ast>> \widetilde{H}^k(E^\alpha) @>\delta>> \cdots.
\end{CD}
\]
By exactness,
\[
\im j^\ast = \ker i^\ast \ne 0,
\]
so $H^k(E^\beta,E^\alpha) \ne 0$.

Finally we show that $H^{k+1}(E^\gamma,E^\beta) \ne 0$. Since $E < \gamma$ on $A^\ast$ and $\beta < \gamma$, we have the inclusions $A \incl A^\ast \incl E^\gamma$ and $A \incl E^\beta \incl E^\gamma$, yielding the commutative diagram
\[
\begin{CD}
\widetilde{H}^k(E^\gamma) @>>> \widetilde{H}^k(A^\ast)\\
@VV{i^\ast}V @VVV\\
\widetilde{H}^k(E^\beta) @>\iota_1^\ast>> \widetilde{H}^k(A).
\end{CD}
\]
The set $A^\ast$ is contractible. Indeed, the mapping
\[
A^\ast \times [0,1] \to A^\ast, \quad (u,t) \mapsto (1 - t)\, u + tr\, \pi_\M(u)
\]
is a strong deformation retraction of $A^\ast$ onto $\set{ru : u \in A_1}$, which is homeomorphic to $A_1$ and hence contractible as in the proof of Theorem \ref{Theorem 3}. So $\widetilde{H}^\ast(A^\ast) = 0$ and hence
\[
\iota_1^\ast i^\ast = 0.
\]
Since $\iota_1^\ast$ is nontrivial, this implies that $i^\ast$ is not surjective. Now consider the exact sequence of the pair $(E^\gamma,E^\beta)$:
\[
\begin{CD}
\cdots @>j^\ast>> \widetilde{H}^k(E^\gamma) @>i^\ast>> \widetilde{H}^k(E^\beta) @>\delta>> H^{k+1}(E^\gamma,E^\beta) @>j^\ast>> \cdots.
\end{CD}
\]
By exactness,
\[
\ker \delta = \im i^\ast \ne \widetilde{H}^k(E^\beta),
\]
so $H^{k+1}(E^\gamma,E^\beta) \ne 0$.
\end{proof}

\section{Proofs of Theorems \ref{Theorem 2}, \ref{Theorem 4}, and \ref{Theorem 5}} \label{Section 3}

In this section we prove Theorems \ref{Theorem 2}, \ref{Theorem 4}, and \ref{Theorem 5}. We begin by proving the following proposition, which shows that
\[
E'(u) = \Ap[u] - \lambda \Bp[u] - \mu f(u) - g(u), \quad u \in W
\]
and hence solutions of equation \eqref{4} coincide with critical points of $E$.

\begin{proposition} \label{Proposition 1}
If $h \in C(W,W^\ast)$ is a potential operator and $H \in C^1(W,\R)$ is its potential satisfying $H(0) = 0$, then
\[
H(u) = \int_0^1 \dualp{h(tu)}{u} dt \quad \forall u \in W.
\]
In particular, $H$ is even if $h$ is odd. If $h$ is $(p - 1)$-homogeneous, then
\[
H(u) = \frac{1}{p} \dualp{h(u)}{u} \quad \forall u \in W
\]
and is $p$-homogeneous.
\end{proposition}

\begin{proof}
We have
\[
H(u) = H(0) + \int_0^1 \frac{d}{dt}\, \big(H(tu)\big)\, dt = \int_0^1 \dualp{H'(tu)}{u} dt = \int_0^1 \dualp{h(tu)}{u} dt.
\]
The last integral equals
\[
\int_0^1 t^{p-1} \dualp{h(u)}{u} dt = \frac{1}{p} \dualp{h(u)}{u}
\]
if $h$ is $(p - 1)$-homogeneous.
\end{proof}

Next we establish two properties of the operator $\Ap$ that follows from the assumption $(A_2)$.

\begin{proposition} \label{Proposition 2}
If $(A_2)$ holds, then
\begin{enumroman}
\item \label{Proposition 2.i} $\Ap$ is strictly monotone: $\dualp{\Ap[u] - \Ap[v]}{u - v} > 0$ for all $u \ne v$ in $W$,
\item \label{Proposition 2.ii} $\Ap$ is of type $(S)$: every sequence $\seq{u_j} \subset W$ such that $u_j \wto u$ and $\dualp{\Ap[u_j]}{u_j - u} \to 0$ has a subsequence that converges strongly to $u$.
\end{enumroman}
\end{proposition}

\begin{proof}
\ref{Proposition 2.i} By $(A_2)$,
\begin{multline*}
\dualp{\Ap[u] - \Ap[v]}{u - v} = \dualp{\Ap[u]}{u} - \dualp{\Ap[u]}{v} - \dualp{\Ap[v]}{u} + \dualp{\Ap[v]}{v}\\[10pt]
\ge \norm{u}^p - \norm{u}^{p-1} \norm{v} - \norm{v}^{p-1} \norm{u} + \norm{v}^p = \left(\norm{u}^{p-1} - \norm{v}^{p-1}\right)\! \big(\norm{u} - \norm{v}\big) \ge 0
\end{multline*}
for all $u, v \in W$. If $\dualp{\Ap[u] - \Ap[v]}{u - v} = 0$, then equality holds throughout and hence $\dualp{\Ap[u]}{v} = \norm{u}^{p-1} \norm{v}$ and $\norm{u} = \norm{v}$. The first equality implies that $\alpha u = \beta v$ for some $\alpha, \beta \ge 0$, not both zero. The second equality then implies that either $u = v = 0$, or $\alpha = \beta > 0$. In the latter case, $u = v$ since $\alpha u = \beta v$.

\ref{Proposition 2.ii} As in the proof of \ref{Proposition 2.i},
\begin{equation} \label{7}
\dualp{\Ap[u_j] - \Ap[u]}{u_j - u} \ge \left(\norm{u_j}^{p-1} - \norm{u}^{p-1}\right)\! \big(\norm{u_j} - \norm{u}\big) \ge 0.
\end{equation}
Since $\dualp{\Ap[u_j]}{u_j - u} \to 0$ and $u_j \wto u$,
\[
\dualp{\Ap[u_j] - \Ap[u]}{u_j - u} = \dualp{\Ap[u_j]}{u_j - u} - \dualp{\Ap[u]}{u_j - u} \to 0,
\]
so \eqref{7} implies that $\norm{u_j} \to \norm{u}$. Then $u_j \to u$ since $W$ is uniformly convex.
\end{proof}

Next we prove Theorem \ref{Theorem 4}.

\begin{proof}[Proof of Theorem \ref{Theorem 4}]
For each $w \in W$, the equation
\[
\Ap[u] = \Bp[w]
\]
has a unique solution $u \in W$. Indeed, a solution can be obtained by minimizing the functional
\[
\Phi(u) = \frac{1}{p} \dualp{\Ap[u]}{u} - \dualp{\Bp[w]}{u}, \quad u \in W,
\]
which is coercive and weakly lower semicontinuous since $\dualp{\Ap[u]}{u} = \norm{u}^p$ by $(A_2)$ and $p > 1$. The solution is unique since $\Ap$ is strictly monotone (see Proposition \ref{Proposition 2} \ref{Proposition 2.i}).

Since both $\Ap$ and $\Bp$ are $(p - 1)$-homogeneous and odd, the map
\[
K : W \to W, \quad w \mapsto u
\]
is linear. Moreover, $K$ is compact. To see this, let $\seq{w_j}$ be a bounded sequence in $W$ and let $u_j = Kw_j$. Since $\Bp$ is a compact operator,
\[
\Ap[u_j] = \Bp[w_j] \to l
\]
for a renamed subsequence and some $l \in W^\ast$. By $(A_2)$,
\[
\norm{u_j}^p = \dualp{\Ap[u_j]}{u_j} = \dualp{\Bp[w_j]}{u_j} \le \dnorm{\Bp[w_j]} \norm{u_j},
\]
which implies that $\seq{u_j}$ is bounded since $p > 1$ and $\seq{\Bp[w_j]}$ is bounded. Since $W$ is reflexive, then a further subsequence of $\seq{u_j}$ converges weakly to some $u \in W$. Then
\[
\dualp{\Ap[u_j]}{u_j - u} = \dualp{\Ap[u_j] - l}{u_j - u} + \dualp{l}{u_j - u} \to 0.
\]
Since $\Ap$ is of type $(S)$ (see Proposition \ref{Proposition 2} \ref{Proposition 2.ii}), then $\seq{u_j}$ has a subsequence that converges strongly to $u$.

Let $w \in W \setminus \set{0}$ and let $u = Kw$. Then $u \ne 0$ since $\dualp{\Ap[u]}{w} = \dualp{\Bp[w]}{w} > 0$ by $(B_2)$. Since $I_p$ is $p$-homogeneous, the radial projection of $u$ on $\M$ is given by
\[
\pi_\M(u) = \frac{u}{I_p(u)^{1/p}}.
\]
Since $J_p$ is $p$-homogeneous, this gives
\begin{equation} \label{9}
\Psi(\pi_\M(u)) = \frac{1}{J_p(\pi_\M(u))} = \frac{I_p(u)}{J_p(u)}.
\end{equation}
We have
\[
I_p(u) = \frac{1}{p} \dualp{\Ap[u]}{u} = \frac{1}{p} \dualp{\Bp[w]}{u} \le \frac{1}{p} \dualp{\Bp[w]}{w}^{(p-1)/p} \dualp{\Bp[u]}{u}^{1/p} = J_p(w)^{(p-1)/p}\, J_p(u)^{1/p}
\]
by $(B_2)$ and
\begin{equation} \label{8}
J_p(w) = \frac{1}{p} \dualp{\Bp[w]}{w} = \frac{1}{p} \dualp{\Ap[u]}{w} \le \frac{1}{p} \norm{u}^{p-1} \norm{w} = I_p(u)^{(p-1)/p}\, I_p(w)^{1/p}
\end{equation}
by $(A_2)$, so
\[
\frac{I_p(u)}{J_p(u)} \le \frac{I_p(w)}{J_p(w)}.
\]
For $w \in \M$, $I_p(w) = 1$ and $1/J_p(w) = \Psi(w)$, so combining this with \eqref{9} gives
\begin{equation} \label{10}
\Psi(\pi_\M(Kw)) \le \Psi(w).
\end{equation}

Let
\[
C_0 = \closure{\pi_\M(K(\Psi^{\lambda_k}))}.
\]
Since $\pi_\M(K(\Psi^{\lambda_k}))$ is contained in the closed set $\Psi^{\lambda_k}$ by \eqref{10}, we have
\begin{equation} \label{11}
\pi_\M(K(\Psi^{\lambda_k})) \subset C_0 \subset \Psi^{\lambda_k}.
\end{equation}
Since $\pi_\M \comp K$ is an odd continuous map on $\Psi^{\lambda_k}$, then
\[
i(\Psi^{\lambda_k}) \le i(\pi_\M(K(\Psi^{\lambda_k}))) \le i(C_0) \le i(\Psi^{\lambda_k})
\]
by the monotonicity of the index (see Proposition \ref{Proposition 3} $(i_2)$). Since $\lambda_k < \lambda_{k+1}$, this together with \eqref{6} gives $i(C_0) = k$. Since $I_p(w) = 1$ and $J_p(w) \ge 1/\lambda_k$ for $w \in \Psi^{\lambda_k}$, \eqref{8} implies that $\closure{K(\Psi^{\lambda_k})} \subset W \setminus \set{0}$. The set $\pi_\M(\closure{K(\Psi^{\lambda_k})})$ is compact since $\Psi^{\lambda_k}$ is bounded, $K$ is compact, and $\pi_\M$ is continuous on $W \setminus \set{0}$. Then so is the closed subset $C_0$.
\end{proof}

We are now ready to prove Theorem \ref{Theorem 2}.

\begin{proof}[Proof of Theorem \ref{Theorem 2}]
We apply Theorem \ref{Theorem 1} to the functional $E$ in \eqref{22}. If $0 < \lambda < \lambda_1$, take $A_0 = \emptyset$ and $B_0 = \M$. If $\lambda \ge \lambda_1$, then $\lambda_k \le \lambda < \lambda_{k+1}$ for some $k \ge 1$. Let $A_0$ be a compact symmetric subset of $\Psi^{\lambda_k}$ of index $k$ (see Theorem \ref{Theorem 4}) and let $B_0 = \Psi_{\lambda_{k+1}}$. Then
\[
i(A_0) = i(\M \setminus B_0) = k
\]
by \eqref{6}.

For $u \in \M$ and $t > 0$,
\begin{equation} \label{17}
E(tu) = t^p \left(1 - \frac{\lambda}{\Psi(u)}\right) - \mu F(tu) - G(tu).
\end{equation}
For $w \in B_0$, this together with $(G_1)$ gives
\[
E(tw) \ge t^p \left(1 - \frac{\lambda}{\lambda_{k+1}} + \o(1)\right) - \mu F(tw) \quad \text{as } t \to 0.
\]
Since $\lambda < \lambda_{k+1}$, it follows from this and $(F_3)$ that $\exists\, \rho, \mu_0 > 0$ such that
\[
\inf_B\, E > 0
\]
for all $0 < \mu < \mu_0$, where $B$ is as in \eqref{16}. Fix $0 < \mu < \mu_0$, let $w_0 \in \M \setminus A_0$, and let $A_1$ be as in \eqref{14}. Since $A_0$ is compact, so is $A_1$. For $u \in A_1$, \eqref{17} together with $(G_2)$ gives
\[
E(tu) \le t^p \left(1 - \mu\, \frac{F(tu)}{t^p}\right).
\]
Since $A_1$ is compact, it follows from this and $(F_1)$ that $\exists\, 0 < r < \rho$ such that
\begin{equation} \label{12}
\sup \set{E(ru) : u \in A_1} < 0.
\end{equation}
Similarly, \eqref{17} together with $(F_2)$ gives
\[
E(tu) \le t^p \left(1 - \frac{G(tu)}{t^p}\right)
\]
for $u \in A_1$, and it follows from this and $(G_4)$ that $\exists\, R > \rho$ such that
\begin{equation}
\sup \set{E(Ru) : u \in A_1} < 0.
\end{equation}
For $v \in A_0$,
\[
E(tv) < - t^p \left(\frac{\lambda}{\Psi(v)} - 1\right) \le 0
\]
since $\Psi(v) \le \lambda_k \le \lambda$. Since $A_0$ is compact, it follows from this that
\begin{equation} \label{13}
\sup \set{E(tv) : v \in A_0,\, r \le t \le R} < 0.
\end{equation}
Combining \eqref{12}--\eqref{13} gives
\[
\sup_A\, E < 0,
\]
where $A$ is as in \eqref{15}. So the second inequality in \eqref{2} holds. The first and the third inequalities hold for some $a < b$ since $A^\ast$ and $B^\ast$ are bounded sets and $E$ is bounded on bounded sets by $(A_2)$, $(B_3)$, $(F_3)$, and $(G_3)$. So $E$ has two nontrivial critical points $u_1, u_2$ with
\[
E(u_1) \le \sup_A\, E < 0 < \inf_B\, E \le E(u_2). \QED
\]
\end{proof}

Finally we prove Theorem \ref{Theorem 5}.

\begin{proof}[Proof of Theorem \ref{Theorem 5}]
We proceed as in the proof of Theorem \ref{Theorem 2}. If $0 < \lambda < \lambda_1$, take $A_0 = \emptyset$ and $B_0 = \M$. Then
\[
A^\ast \subset \set{tw_0 : t \ge 0}
\]
and hence $\sup E(A^\ast) < c_\mu$ for all sufficiently small $\mu > 0$ by \eqref{23}. If $\lambda_k \le \lambda < \lambda_{k+1}$, let $A_0 = C$ and $B_0 = \Psi_{\lambda_{k+1}}$. Then
\[
A^\ast \subset \set{sv + tw_0 : v \in C,\, s, t \ge 0}
\]
and hence $\sup E(A^\ast) < c_\mu$ for all sufficiently small $\mu > 0$ by \eqref{24}. So we can apply Theorem \ref{Theorem 1} with $a < \inf E(B^\ast)$ and $b = c_\mu$ to conclude the proof.
\end{proof}

\section{Proof of Theorem \ref{Theorem 6}} \label{Section 4}

In this section we prove Theorem \ref{Theorem 6} by applying Theorem \ref{Theorem 5} with $W = W^{1,\,p}_0(\Omega)$ and the operators $\Ap, \Bp, f, g \in C(W^{1,\,p}_0(\Omega),W^{-1,\,p'}(\Omega))$ given by
\begin{multline*}
\dualp{\Ap[u]}{v} = \int_\Omega |\nabla u|^{p-2}\, \nabla u \cdot \nabla v\, dx, \quad \dualp{\Bp[u]}{v} = \int_\Omega |u|^{p-2}\, uv\, dx,\\[5pt]
\dualp{f(u)}{v} = \int_\Omega f(x,u)\, v\, dx, \quad \dualp{g(u)}{v} = \int_\Omega |u|^{p^\ast - 2}\, uv\, dx, \quad u, v \in W^{1,\,p}_0(\Omega).
\end{multline*}
We begin by determining a threshold level below which the functional $E$ satisfies the \PS{} condition.

\begin{lemma} \label{Lemma 1}
Let $0 < \mu \le 1$. Then $\exists\, \kappa > 0$ such that $E$ satisfies the {\em \PS{c}} condition for all
\begin{equation} \label{52}
c < \frac{1}{N}\, S^{N/p} - \kappa \mu.
\end{equation}
\end{lemma}

\begin{proof}
Let $c \in \R$ and let $\seq{u_j}$ be a sequence in $W^{1,\,p}_0(\Omega)$ such that
\begin{equation} \label{42}
E(u_j) = \int_\Omega \left(\frac{1}{p}\, |\nabla u_j|^p - \frac{\lambda}{p}\, |u_j|^p - \mu F(x,u_j) - \frac{1}{p^\ast}\, |u_j|^{p^\ast}\right) dx = c + \o(1)
\end{equation}
and
\begin{multline} \label{43}
\dualp{E'(u_j)}{v} = \int_\Omega \left(|\nabla u_j|^{p-2}\, \nabla u_j \cdot \nabla v - \lambda\, |u_j|^{p-2}\, u_j\, v - \mu f(x,u_j)\, v - |u_j|^{p^\ast - 2}\, u_j\, v\right) dx\\[7.5pt]
= \o(\norm{v}) \quad \forall v \in W^{1,\,p}_0(\Omega).
\end{multline}
Taking $v = u_j$ in \eqref{43} gives
\begin{equation} \label{44}
\int_\Omega \left(|\nabla u_j|^p - \lambda\, |u_j|^p - \mu f(x,u_j)\, u_j - |u_j|^{p^\ast}\right) dx = \o(\norm{u_j}).
\end{equation}
Since $r < p^\ast$, \eqref{42} and \eqref{44} imply that $\seq{u_j}$ is bounded, so a renamed subsequence converges to some $u$ weakly in $W^{1,\,p}_0(\Omega)$, strongly in $L^s(\Omega)$ for all $s \in [1,p^\ast)$, and a.e.\! in $\Omega$. Setting $\widetilde{u}_j = u_j - u$, we will show that $\widetilde{u}_j \to 0$ in $W^{1,\,p}_0(\Omega)$.

Equation \eqref{44} implies
\begin{equation} \label{48}
\norm{u_j}^p = \pnorm[p^\ast]{u_j}^{p^\ast} + \int_\Omega \left(\lambda\, |u|^p + \mu f(x,u)\, u\right) dx + \o(1),
\end{equation}
where $\pnorm[p^\ast]{\cdot}$ denotes the $L^{p^\ast}(\Omega)$-norm. Taking $v = u$ in \eqref{43} and passing to the limit gives
\begin{equation} \label{46}
\norm{u}^p = \pnorm[p^\ast]{u}^{p^\ast} + \int_\Omega \left(\lambda\, |u|^p + \mu f(x,u)\, u\right) dx.
\end{equation}
Since
\begin{equation} \label{54}
\norm{\widetilde{u}_j}^p = \norm{u_j}^p - \norm{u}^p + \o(1)
\end{equation}
and
\[
\pnorm[p^\ast]{\widetilde{u}_j}^{p^\ast} = \pnorm[p^\ast]{u_j}^{p^\ast} - \pnorm[p^\ast]{u}^{p^\ast} + \o(1)
\]
by the Br{\'e}zis-Lieb lemma \cite[Theorem 1]{MR699419}, \eqref{48} and \eqref{46} imply
\[
\norm{\widetilde{u}_j}^p = \pnorm[p^\ast]{\widetilde{u}_j}^{p^\ast} + \o(1) \le \frac{\norm{\widetilde{u}_j}^{p^\ast}}{S^{p^\ast/p}} + \o(1),
\]
so
\begin{equation} \label{50}
\norm{\widetilde{u}_j}^p \left(S^{N/(N-p)} - \norm{\widetilde{u}_j}^{p^2/(N-p)}\right) \le \o(1).
\end{equation}
On the other hand, \eqref{42} implies
\[
c = \frac{1}{p} \norm{u_j}^p - \frac{1}{p^\ast} \pnorm[p^\ast]{u_j}^{p^\ast} - \int_\Omega \left(\frac{\lambda}{p}\, |u|^p + \mu F(x,u)\right) dx + \o(1),
\]
and a straightforward calculation combining this with \eqref{48}--\eqref{54} gives
\begin{equation} \label{53}
c = \frac{1}{N} \norm{\widetilde{u}_j}^p + \int_\Omega K(x,u)\, dx + \o(1),
\end{equation}
where
\[
K(x,t) = \frac{1}{N}\, |t|^{p^\ast} + \mu \left(\frac{1}{p}\, f(x,t)\, t - F(x,t)\right).
\]
By \eqref{27},
\[
\frac{1}{p}\, f(x,t)\, t - F(x,t) \ge - a_1\, (|t|^r + 1) \quad \text{for a.a.\! } x \in \Omega \text{ and all } t \in \R
\]
for some $a_1 > 0$. Since $r < p^\ast$ and $\mu \le 1$, this together with the Young's inequality gives
\[
K(x,t) \ge - a_2\, \mu \quad \text{for a.a.\! } x \in \Omega \text{ and all } t \in \R
\]
for some $a_2 > 0$. Then \eqref{53} gives
\[
\norm{\widetilde{u}_j}^p \le N(c + \kappa \mu) + \o(1)
\]
for some $\kappa > 0$. Combining this with \eqref{50} shows that $\widetilde{u}_j \to 0$ when \eqref{52} holds.
\end{proof}

We will apply Theorem \ref{Theorem 5} with
\[
c_\mu = \frac{1}{N}\, S^{N/p} - \kappa \mu,
\]
where $\kappa > 0$ is as in Lemma \ref{Lemma 1}. We only consider the case where $\lambda_k < \lambda < \lambda_{k+1}$ for some $k \ge 1$ since the case where $0 < \lambda < \lambda_1$ is similar and simpler. We have $\M = \set{u \in W^{1,\,p}_0(\Omega) : \norm{u}^p = p}$ and $\Psi(u) = p/\int_\Omega |u|^p\, dx$ for $u \in \M$. We need to show that there exist a compact symmetric subset $C$ of $\Psi^\lambda$ with $i(C) = k$ and $w_0 \in \M \setminus C$ such that
\begin{equation} \label{29}
\sup_{v \in C,\, s, t \ge 0}\, E(sv + tw_0) < \frac{1}{N}\, S^{N/p} - \kappa \mu
\end{equation}
for all sufficiently small $\mu > 0$.

By iterating a sufficient number of times the map $K$ used in the proof of Theorem \ref{Theorem 4}, we may assume that the compact symmetric subset $C_0$ of $\Psi^{\lambda_k}$ with $i(C_0) = k$ constructed in that theorem is bounded in $L^\infty(\Omega) \cap C^{1,\,\alpha}_\loc(\Omega)$ (see Degiovanni and Lancelotti \cite[Theorem 2.3]{MR2514055}). We may assume without loss of generality that $0 \in \Omega$. Let $\delta_0 = \dist{0}{\bdry{\Omega}}$, let $\eta : [0,\infty) \to [0,1]$ be a smooth function such that $\eta(s) = 0$ for $s \le 3/4$ and $\eta(s) = 1$ for $s \ge 1$, and set
\[
u_\delta(x) = \eta\bigg(\frac{|x|}{\delta}\bigg)\, u(x), \quad 0 < \delta \le \delta_0/2
\]
for $u \in C_0$. Then set
\[
v = \pi_\M(u_\delta),
\]
where $\pi_\M : W^{1,\,p}_0(\Omega) \setminus \set{0} \to \M,\, u \mapsto p^{1/p}\, u/\norm{u}$ is the radial projection onto $\M$, and let
\[
C = \set{v : u \in C_0}.
\]

\begin{lemma} \label{Lemma 2}
If $\delta > 0$ is sufficiently small, then $C$ is a compact symmetric subset of $\Psi^\lambda$ with $i(C) = k$.
\end{lemma}

\begin{proof}
First we show that $C \subset \Psi^\lambda$ if $\delta > 0$ is sufficiently small. Let $u \in C_0$. Since functions in $C_0$ are bounded in $C^1(B_{\delta_0/2}(0))$ and belong to $\Psi^{\lambda_k}$,
\[
\int_\Omega |\nabla u_\delta|^p\, dx \le \int_{\Omega \setminus B_\delta(0)} |\nabla u|^p\, dx + \int_{B_\delta(0)} \left(|\nabla u|^p + a_1\, \delta^{-p}\, |u|^p\right) dx \le p + a_2\, \delta^{N-p}
\]
for some $a_1, a_2 > 0$ and
\[
\int_\Omega |u_\delta|^p\, dx \ge \int_{\Omega \setminus B_\delta(0)} |u|^p\, dx = \int_\Omega |u|^p\, dx - \int_{B_\delta(0)} |u|^p\, dx \ge \frac{p}{\lambda_k} - a_3\, \delta^N
\]
for some $a_3 > 0$. So
\[
\Psi(v) = \frac{\dint_\Omega |\nabla u_\delta|^p\, dx}{\dint_\Omega |u_\delta|^p\, dx} \le \lambda_k + a_4\, \delta^{N-p}
\]
for some $a_4 > 0$. Since $\lambda_k < \lambda$, the last expression is less than or equal to $\lambda$, and hence $v \in \Psi^\lambda$, for all sufficiently small $\delta > 0$.

Since $C_0$ is a compact symmetric set and $u \mapsto v$ is an odd continuous map of $C_0$ onto $C$, $C$ is also a compact symmetric set and
\[
i(C) \ge i(C_0) = k
\]
by the monotonicity of the index (see Proposition \ref{Proposition 3} $(i_2)$). On the other hand, since $C \subset \Psi^\lambda \subset \M \setminus \Psi_{\lambda_{k+1}}$,
\[
i(C) \le i(\M \setminus \Psi_{\lambda_{k+1}}) = k
\]
by \eqref{6}. So $i(C) = k$.
\end{proof}

We are now ready to prove Theorem \ref{Theorem 6}.

\begin{proof}[Proof of Theorem \ref{Theorem 6}]
Recall that the infimum in \eqref{28} is attained by the Aubin-Talenti functions
\[
u^\ast_\eps(x) = \frac{c_{N,p}\, \eps^{(N-p)/p^2}}{\big(\eps + |x|^{p/(p-1)}\big)^{(N-p)/p}}, \quad \eps > 0
\]
when $\Omega = \R^N$, where the constant $c_{N,p} > 0$ is chosen so that
\[
\int_{\R^N} |\nabla u^\ast_\eps|^p\, dx = \int_{\R^N} |u^\ast_\eps|^{p^\ast} dx = S^{N/p}.
\]
Fix $\delta > 0$ so small that $C$ is a compact symmetric subset of $\Psi^\lambda$ with $i(C) = k$ (see Lemma \ref{Lemma 2}), let $\theta : [0,\infty) \to [0,1]$ be a smooth function such that $\theta(s) = 1$ for $s \le 1/4$ and $\theta(s) = 0$ for $s \ge 1/2$, and set
\[
u_\eps(x) = \theta\bigg(\frac{|x|}{\delta}\bigg)\, u^\ast_\eps(x), \quad \widetilde{u}_\eps(x) = \frac{u_\eps(x)}{\left(\dint_{\R^N} |u_\eps|^{p^\ast} dx\right)^{1/p^\ast}}, \quad \eps > 0.
\]
We have the well-known estimates
\begin{gather}
\label{31} \int_{\R^N} |\nabla \widetilde{u}_\eps|^p\, dx \le S + a_1\, \eps^{(N-p)/p},\\[10pt]
\int_{\R^N} |\widetilde{u}_\eps|^{p^\ast} dx = 1,\\[10pt]
\label{32} \int_{\R^N} |\widetilde{u}_\eps|^p\, dx \ge \begin{cases}
a_2\, \eps^{p-1} & \text{if } N > p^2\\[7.5pt]
a_2\, \eps^{p-1} \abs{\log \eps} & \text{if } N = p^2
\end{cases}
\end{gather}
for some $a_1, a_2 > 0$ (see, e.g., Dr{\'a}bek and Huang \cite{MR1473856}). Let
\[
w_0 = \pi_\M(\widetilde{u}_\eps).
\]
Since functions in $C$ have their supports in $\Omega \setminus B_{3 \delta/4}(0)$, while the support of $w_0$ is in $\closure{B_{\delta/2}(0)}$, $w_0 \in \M \setminus C$. We will show that \eqref{29} holds for all sufficiently small $\eps, \mu > 0$.

Let $v \in C$ and $s, t \ge 0$. Since $v$ and $w_0$ have disjoint supports,
\begin{equation} \label{33}
E(sv + tw_0) = E(sv) + E(tw_0).
\end{equation}
By \eqref{30},
\begin{equation}
E(sv) \le \frac{s^p}{p} \left(\int_\Omega |\nabla v|^p\, dx - \lambda \int_\Omega |v|^p\, dx\right) = - s^p \left(\frac{\lambda}{\Psi(v)} - 1\right) \le 0
\end{equation}
since $v \in \Psi^\lambda$. Moreover,
\[
E(tw_0) \le \frac{t^p}{p} \left(\int_\Omega |\nabla w_0|^p\, dx - \lambda \int_\Omega |w_0|^p\, dx\right) - \frac{t^{p^\ast}}{p^\ast} \int_\Omega |w_0|^{p^\ast} dx,
\]
and maximizing the right-hand side over all $t \ge 0$ gives
\begin{eqnarray} \label{34}
E(tw_0) & \le & \frac{1}{N} \frac{\left(\dint_\Omega |\nabla w_0|^p\, dx - \lambda \dint_\Omega |w_0|^p\, dx\right)^{p^\ast/(p^\ast - p)}}{\left(\dint_\Omega |w_0|^{p^\ast} dx\right)^{p/(p^\ast - p)}} \notag\\[5pt]
& = & \frac{1}{N} \frac{\left(\dint_\Omega |\nabla \widetilde{u}_\eps|^p\, dx - \lambda \dint_\Omega |\widetilde{u}_\eps|^p\, dx\right)^{p^\ast/(p^\ast - p)}}{\left(\dint_\Omega |\widetilde{u}_\eps|^{p^\ast} dx\right)^{p/(p^\ast - p)}} \notag\\[7.5pt]
& \le & \frac{1}{N} \begin{cases}
\left(S + a_1\, \eps^{(N-p)/p} - \lambda a_2\, \eps^{p-1}\right)^{N/p} & \text{if } N > p^2\\[7.5pt]
\left(S + a_1\, \eps^{p-1} - \lambda a_2\, \eps^{p-1} \abs{\log \eps}\right)^p & \text{if } N = p^2
\end{cases}
\end{eqnarray}
by \eqref{31}--\eqref{32}. It follows from \eqref{33}--\eqref{34} that
\[
\sup_{v \in C,\, s, t \ge 0}\, E(sv + tw_0) < \frac{1}{N}\, S^{N/p}
\]
if $\eps > 0$ is sufficiently small. Then \eqref{29} holds if, in addition, $\mu > 0$ is sufficiently small.
\end{proof}

\section{Proof of Theorem \ref{Theorem 7}} \label{Section 5}

In this section we prove Theorem \ref{Theorem 7} by applying Theorem \ref{Theorem 5} with $W = W^{s,\,p}_0(\Omega)$ and the operators $\Ap$, $\Bp$, $f$, and $g$ on $W^{s,\,p}_0(\Omega)$ given by
\begin{multline*}
\dualp{\Ap[u]}{v} = \int_{\R^{2N}} \frac{|u(x) - u(y)|^{p-2}\, (u(x) - u(y))\, (v(x) - v(y))}{|x - y|^{N+sp}}\, dx dy,\\[5pt]
\dualp{\Bp[u]}{v} = \int_\Omega |u|^{p-2}\, uv\, dx, \quad \dualp{f(u)}{v} = \int_\Omega f(x,u)\, v\, dx, \quad \dualp{g(u)}{v} = \int_\Omega |u|^{p^\ast - 2}\, uv\, dx,\\[7.5pt]
u, v \in W^{s,\,p}_0(\Omega).
\end{multline*}
The proof is similar to that of Theorem \ref{Theorem 6}, so we will be sketchy.

\begin{proof}[Proof of Theorem \ref{Theorem 7}]
An argument similar to that in the proof of Lemma \ref{Lemma 1} shows that if $0 < \mu \le 1$, then $\exists\, \kappa > 0$ such that $E$ satisfies the \PS{c} condition for all
\[
c < \frac{s}{N}\, S^{N/sp} - \kappa \mu.
\]
We will apply Theorem \ref{Theorem 5} with
\[
c_\mu = \frac{s}{N}\, S^{N/sp} - \kappa \mu.
\]
We only consider the case where $\lambda_k < \lambda < \lambda_{k+1}$ for some $k \ge 1$ since the case where $0 < \lambda < \lambda_1$ is similar and simpler. We have $\M = \set{u \in W^{s,\,p}_0(\Omega) : \norm{u}^p = p}$ and $\Psi(u) = p/\int_\Omega |u|^p\, dx$ for $u \in \M$. We need to show that there exist a compact symmetric subset $C$ of $\Psi^\lambda$ with $i(C) = k$ and $w_0 \in \M \setminus C$ such that
\begin{equation} \label{41}
\sup_{v \in C,\, s, t \ge 0}\, E(sv + tw_0) < \frac{s}{N}\, S^{N/sp} - \kappa \mu
\end{equation}
for all sufficiently small $\mu > 0$.

In the absence of an explicit formula for a minimizer for $S$ in \eqref{36}, we will use certain asymptotic estimates for minimizers obtained in Brasco et al.\! \cite{MR3461371}. It was shown there that there exists a nonnegative, radially symmetric, and decreasing minimizer $U = U(r)$ satisfying
\begin{equation} \label{38}
\int_{\R^{2N}} \frac{|U(x) - U(y)|^p}{|x - y|^{N+sp}}\, dx dy = \int_{\R^N} |U|^{p_s^\ast}\, dx = S^{N/sp}
\end{equation}
and
\begin{equation} \label{37}
c_1 r^{-(N-sp)/(p-1)} \le U(r) \le c_2 r^{-(N-sp)/(p-1)} \quad \forall r \ge 1
\end{equation}
for some constants $c_1, c_2 > 0$. By \eqref{37},
\[
\frac{U(\theta r)}{U(r)} \le \frac{c_2}{c_1}\, \theta^{-(N-sp)/(p-1)} \le \half \quad \forall r \ge 1
\]
if $\theta > 1$ is a sufficiently large constant. The function
\[
U_\eps(x) = \eps^{-(N-sp)/p}\, U\bigg(\frac{|x|}{\eps}\bigg)
\]
is also a minimizer for $S$ satisfying \eqref{38} for any $\eps > 0$. For $\eps, \delta > 0$, let
\[
m_{\eps,\delta} = \frac{U_\eps(\delta)}{U_\eps(\delta) - U_\eps(\theta \delta)},
\]
let
\[
g_{\eps,\delta}(t) = \begin{cases}
0 & \text{if } 0 \le t \le U_\eps(\theta \delta)\\[7.5pt]
m_{\eps,\delta}^p\, (t - U_\eps(\theta \delta)) & \text{if } U_\eps(\theta \delta) \le t \le U_\eps(\delta)\\[7.5pt]
t + U_\eps(\delta)\, (m_{\eps,\delta}^{p-1} - 1) & \text{if } t \ge U_\eps(\delta),
\end{cases}
\]
let
\[
G_{\eps,\delta}(t) = \int_0^t g_{\eps,\delta}'(\tau)^{1/p}\, d\tau = \begin{cases}
0 & \text{if } 0 \le t \le U_\eps(\theta \delta)\\[7.5pt]
m_{\eps,\delta}\, (t - U_\eps(\theta \delta)) & \text{if } U_\eps(\theta \delta) \le t \le U_\eps(\delta)\\[7.5pt]
t & \text{if } t \ge U_\eps(\delta),
\end{cases}
\]
and set
\[
u_{\eps,\delta}(r) = G_{\eps,\delta}(U_\eps(r)).
\]
For $\eps \le \delta/2$, we have the estimates
\begin{gather}
\label{39} \int_{\R^{2N}} \frac{|u_{\eps,\delta}(x) - u_{\eps,\delta}(y)|^p}{|x - y|^{N+sp}}\, dx dy \le S^{N/sp} + a_1 \Big(\frac{\eps}{\delta}\Big)^{(N-sp)/(p-1)},\\[10pt]
\int_{\R^N} |u_{\eps,\delta}|^{p_s^\ast}\, dx \ge S^{N/sp} - a_1 \Big(\frac{\eps}{\delta}\Big)^{N/(p-1)},\\[10pt]
\label{40} \int_{\R^N} |u_{\eps,\delta}|^p\, dx \ge \begin{cases}
a_2\, \eps^{sp} & \text{if } N > sp^2\\[7.5pt]
a_2\, \eps^{sp} \abs{\log \Big(\dfrac{\eps}{\delta}\Big)} & \text{if } N = sp^2
\end{cases}
\end{gather}
for some $a_1, a_2 > 0$ (see Mosconi et al.\! \cite[Lemma 2.7]{MR3530213}). Let
\[
w_0 = \pi_\M(u_{\eps,\delta}),
\]
where $\pi_\M : W^{s,\,p}_0(\Omega) \setminus \set{0} \to \M,\, u \mapsto p^{1/p}\, u/\norm{u}$ is the radial projection onto $\M$.

By iterating a sufficient number of times the map $K$ used in the proof of Theorem \ref{Theorem 4}, we may assume that the compact symmetric subset $C_0$ of $\Psi^{\lambda_k}$ with $i(C_0) = k$ constructed in that theorem consists of functions $u$ such that $(- \Delta)_p^s\, u$ is bounded in $L^\infty(\Omega)$ (see Mosconi et al.\! \cite[Proposition 3.1]{MR3530213}). We may assume without loss of generality that $0 \in \Omega$. Let $\eta : \R^N \to [0,1]$ be a smooth function such that $\eta(x) = 0$ for $|x| \le 2 \theta$ and $\eta(x) = 1$ for $|x| \ge 3 \theta$, and set
\[
u_\delta(x) = \eta\Big(\frac{x}{\delta}\Big)\, u(x), \quad \delta > 0
\]
for $u \in C_0$. Then set
\[
v = \pi_\M(u_\delta)
\]
and let
\[
C = \set{v : u \in C_0}.
\]
If $\delta > 0$ is sufficiently small, then $C$ is a compact symmetric subset of $\Psi^\lambda$ with $i(C) = k$ (see Mosconi et al.\! \cite[Proposition 3.2]{MR3530213}). Since functions in $C$ have their supports in $\Omega \setminus B_{2 \theta \delta}(0)$, while the support of $w_0$ is in $\closure{B_{\theta \delta}(0)}$, $w_0 \in \M \setminus C$. An argument similar to that in the proof of Theorem \ref{Theorem 6} shows that
\[
\sup_{v \in C,\, s, t \ge 0}\, E(sv + tw_0) < \frac{s}{N}\, S^{N/sp}
\]
if $\eps > 0$ is sufficiently small (see the proof of Theorem 1.3 in Mosconi et al.\! \cite{MR3530213}). Then \eqref{41} holds if, in addition, $\mu > 0$ is sufficiently small. \end{proof}

\def\cdprime{$''$}

\end{document}